\documentclass[english,11pt]{article}
\usepackage{color}
\usepackage{verbatim}
\usepackage{float}
\usepackage{mathtools}
\usepackage[inline,shortlabels]{enumitem}
\usepackage{fullpage}
\usepackage{todonotes}
\usepackage{amsmath}
\usepackage{amsthm}
\usepackage{amssymb}
\usepackage{xargs}[2008/03/08]

\colorlet{darkred}{red!80!black}
\colorlet{darkblue}{blue!80!black}
\colorlet{darkgreen}{green!60!black}

\usepackage[noend]{algpseudocode}
\floatstyle{ruled}
\newfloat{algorithm}{tbp}{loa}
\floatname{algorithm}{Procedure}

\theoremstyle{plain}
\newtheorem{thm}{Theorem}
\theoremstyle{definition}
\newtheorem{defn}[thm]{Definition}
\theoremstyle{remark}

\theoremstyle{plain}
\newtheorem{cor}[thm]{Corollary}
\newtheorem{observation}[thm]{Observation}
\theoremstyle{remark}
\newtheorem{rem}[thm]{Remark}
\theoremstyle{plain}
\newtheorem{lem}[thm]{Lemma}
\theoremstyle{plain}

\theoremstyle{plain}

\usepackage[left,modulo]{lineno}

\newcommand{\W}[2]{\mathcal W_{#1}(#2)}
\newcommand{\Hb}[1]{H_{B}(#1)}
\newcommand{\Hr}[1]{H_{R}(#1)}
\newcommand{\0}{\ensuremath\mathtt{wlk}}
\newcommand{\1}{\ensuremath\mathtt{trl}}
\newcommand{\2}{\ensuremath\mathtt{pth}}
\newcommand{\3}{\ensuremath\mathtt{ind}}
\newcommand{\4}{\ensuremath\mathtt{iso}}

\renewcommand{\S}[2]{\xrightleftharpoons[{#1}]{#2} }
\newcommandx\s[1][usedefault, addprefix=\global, 1=]{\underset{#1}{arrow}}%

\def\mpfile#1#2{\includegraphics{#1#2.pdf}}

\usepackage{babel}
\title{Shifting paths to avoidable ones}
\author{
Vladimir Gurvich\thanks{National Research University Higher School of Economics, \texttt{vladimir.gurvich@gmail.com}.}\and
Matjaž Krnc\thanks{FAMNIT, University of Primorska, \texttt{matjaz.krnc@upr.si}. Corresponding author.}\and
Martin Milanič\thanks{FAMNIT and IAM, University of Primorska, \texttt{martin.milanic@upr.si}.}\and
Mikhail Vyalyi\thanks{National Research University Higher School of Economics; Moscow Institute of Physics and Technology; Federal Research Center “Computer Science and Control” of the Russian Academy of Science, \texttt{vyalyi@gmail.com}.}}
\begin{document}
\maketitle

\begin{abstract}
An extension of an induced path $P$ in a graph $G$ is an induced path $P'$ such that deleting the endpoints of $P'$ results in $P$. An induced path in a graph is said to be avoidable if each of its extensions is contained in an induced cycle. In 2019, Beisegel, Chudovsky, Gurvich, Milanič, and Servatius conjectured that every graph that contains an induced $k$-vertex path also contains an avoidable induced path of the same length, and proved the result for $k = 2$.
The case $k = 1$ was known much earlier, due to a work of Ohtsuki, Cheung, and Fujisawa in 1976.
The conjecture was proved for all $k$ in 2020 by Bonamy, Defrain, Hatzel, and Thiebaut.
In the present paper, using a similar approach, we strengthen their result
from a reconfiguration point of view. Namely, we show that in every graph, each induced path can be transformed to an avoidable one by a sequence of shifts,
where two induced $k$-vertex paths are shifts of each other if their union is an induced path with $k+1$ vertices.
We also obtain analogous results for not necessarily induced paths and for walks.
In contrast, the statement cannot be extended to trails or to isometric paths.


\smallskip
{\bf MSC codes (2020):} 05C38 (primary), 05C12, 05C05, 05C76 (secondary)
\end{abstract}

\section{Introduction}

All graphs considered in this paper will be finite, undirected, and may have loops and multiple edges, unless stated otherwise (in which case the graph will be referred to as a simple graph). We consider five types of walks in graphs: general walks, trails, paths, induced paths, and isometric paths. We follow the terminology used in~\cite{west2001introduction}. Given a non-negative integer $\ell$, a \emph{$v_{0},v_{\ell}$-walk of length $\ell$} in a graph $G$ is a sequence $(v_{0},e_{1},v_{1},\ldots,e_{\ell},v_{\ell})$, where $v_{0},\ldots,v_{\ell}\in V(G)$, $e_{1},\ldots,e_{\ell}\in E(G)$, and for all $i\in\{1,\ldots,\ell\}$ edge $e_{i}$ has endpoints $v_{i-1}$ and $v_{i}$.
If $v_{0}=v_{\ell}$, the walk is said to be \emph{closed}.
A walk in which all edges (resp.~vertices) are distinct is a \emph{trail} (resp.~a \emph{path}) in $G$.

A subgraph $H$ of a graph $G$ is an \emph{induced} subgraph of $G$ if the set of edges of $H$ is exactly the set of edges of $G$ having both endpoints in $V(H)$. The \emph{distance} between two vertices $u$ and $v$ in a graph $G$ is denoted by $d_{G}(u,v)$ and defined as the length of a shortest $u,v$-path in $G$ (or $\infty$
if there is no $u,v$-path in $G$).
A subgraph $H$ of $G$ is said to be \emph{isometric in $G$} if $d_{H}(u,v)=d_{G}(u,v)$ for every two vertices $u,v\in V(H)$.

Note that a path $P$ in a graph $G$ can be viewed as a subgraph of $G$ (with a pair of mutually inverse paths yielding the same subgraph).
In particular, we say that a path in $G$ is \emph{induced} if the corresponding subgraph is induced in $G$, and \emph{isometric} if the corresponding subgraph is isometric in $G$.
For a positive integer $k$ we denote by $P_{k}$ the graph corresponding to a $k$-vertex path (without a host graph $G$), that is,
the graph with $k$ vertices $v_1,\ldots, v_k$
and $\ell=k-1$ edges $\{\{v_{i},v_{i+1}\}\mid i\in \{1,\ldots,\ell\}\}$.

\subsection{Five types of walks}

We consider the following five types of walks:
\begin{table}[!h]
\centerline{\begin{tabular}{|c|c|c|c|c|c|}
\hline
$t$ & $\0$ & $\1$ & $\2$ & $\3$ & $\4$\\
\hline
walk of type $t$ &
walk & trail& path & induced path & isometric path\\
\hline
\end{tabular}
}
\end{table}

{For a graph $G$ and $t\in \{\0,\1,\2,\3,\4\}$, a \emph{$t$-walk} in $G$ is a walk in $G$ of type $t$.
We denote the set of all $t$-walks in $G$ by $\W {t}G$.}
Note that
$$\W {\0}G\supseteq\W {\1}G\supseteq\W {\2}G\supseteq\W {\3}G\supseteq\W {\4}G.$$
Moreover, for $k\in\{0,1\}$, equalities hold if we restrict ourselves to walks of length $k$.
(Note however that for $k=1$ the graph should not contain loops.)

\begin{sloppypar}
\begin{defn}[{Extension of a $t$-walk}]
Let $t\in\{\0,\1,\2,\3,\4\}$ and let $W,W'$ be two {\hbox{$t$-walks}} in a graph $G$.
Let $W=(v_{1},e_{1},v_{2},\ldots,e_{k-1},v_{k})$
for some vertices $v_{1},\ldots, v_k\in V(G)$ and edges $e_{1},\ldots, e_{k-1}\in E(G)$.
We say that $W'$ is a \emph{$t$-extension} of $W$
if $W'=(v_{0},e_{0},v_{1},e_1,\ldots,e_{k-1},v_{k}, e_k, v_{k+1})$ for some vertices $v_0,v_{k+1}\in V(G)$ and edges
$e_0,e_k\in E(G)$.
\end{defn}
\end{sloppypar}

A vertex $v$ in a graph $G$ is said to be \emph{simplicial} if its neighborhood forms a clique.
Note that $v\in V(G)$ is simplicial if and only if the corresponding one-vertex induced path $(v)\in\W \3G$ has no $\3$-extension.
{Among other things, this concept is generalized in the following definition.}

{
\begin{defn}[Simplicial, closable, and avoidable $t$-walk]
Let $t\in\{\0,\1,\2,\3,\4\}$ and let $W$ a be a $t$-walk in a graph $G$. We say that $W$ is:
\begin{itemize}
\item \emph{$t$-simplicial} if it has no $t$-extension,
\item \emph{$t$-closable} if it is a subwalk of a closed $t$-walk in $G$,
\item \emph{$t$-avoidable} in $G$ if every $t$-extension of $W$ is $t$-closable.
\end{itemize}
In particular, every $t$-simplicial $t$-walk is $t$-avoidable.
\end{defn}}

\begin{defn}[Shift of a {$t$-walk}]
Let $t\in\{\0,\1,\2,\3,\4\}$ and {$W$ be a $t$-walk in $G$} having at least one edge.
Let $W=(v_{0},e_{1},v_{1},\ldots,e_{k},v_{k})$ for some $k\ge 1$, vertices $v_0,\ldots, v_k\in V(G)$, and edges
$e_1,\ldots,e_k\in E(G)$.
We say that {$t$}-walks
$W'=(v_{0},e_{1},v_{1},\ldots,v_{k-1})$
and
$W''=(v_{1},\ldots,v_{k-1},e_k,v_k)$
are {$t$-}\emph{shifts} of each other in $G$.
\end{defn}

Furthermore, given two $t$-walks $W$ and $W'$ in $G$, we say that $W$ \emph{can be $t$-shifted in $G$ to} $W'$ if there exists a sequence of $t$-walks $W=W_{0},W_{1},\ldots,W_{p}=W'$ in $G$ such that
for all $j\in\{1,\ldots,p\}$ we have $W_{j}\in\W tG$ and $W_{j}$
is a {$t$-shift} of $W_{j-1}$ in $G$.
Note that $p=0$ is allowed (in which case $W=W'$). We write $W\S{G}{t}W'$ if $W$ can be $t$-shifted to $W'$
in $G$. Note that for every graph $G$, the relation $\S{G}{t}$ is an equivalence relation on the set $\W tG$.
Whenever for some graph $G$ the type $t$ of walks under consideration is clear from context, {we just write $\S{G}{~}$ and talk about ``shifts'' instead of ``$t$-shifts'', about ``extensions of an induced path'' instead of ``$\3$-extensions of an $\3$-walk'', etc.}

\subsection{Main results}
Our main result is given by the following theorem.

\begin{thm}
\label{thm:main} Every walk, path, or induced path in a graph
can be shifted to an avoidable one.
\end{thm}

We prove Theorem~\ref{thm:main} in parts.
The statement for walks follows from Observation~\ref{obs:walks-1} in Section~\ref{sec:walks}.
The statements for induced paths and paths are Theorems~\ref{thm:avoidable-induced}~and~\ref{thm:paths} in Sections~\ref{sec:induced-paths} and~\ref{sec:paths}, respectively.

\begin{cor}
\label{cor:main} For every non-negative integer $\ell$
every graph:
\begin{itemize}
\item[\fbox{$\0$}] either contains no walk of length $\ell$, or contains an avoidable
walk of length~$\ell$;
\item[\fbox{$\2$}] either contains no path of length $\ell$, or contains an avoidable
path of length~$\ell$;
\item[\fbox{$\3$}] either contains no induced path of length $\ell$, or contains an avoidable
induced path of length~$\ell$.
\end{itemize}
\end{cor}
Note that every graph with at least one edge contains walks of all non-negative lengths.

On the other hand, we show that statements of Theorem~\ref{thm:main} and Corollary~\ref{cor:main}
do not extend to the cases of trails and of isometric paths.

\subsection{Related work}

The most important case in Corollary~\ref{cor:main} is the case of induced paths.
The corresponding statement was conjectured (and proved for   $\ell =  1$)  by Beisegel et al.~in~\cite{BCGMS19}.
We also refer to~\cite{BCGMS19} for motivation and more details.
For $\ell=0$ the result is much older; it follows from a work of Ohtsuki et al.~\cite{OCF1976}, see also~\cite{RTL1976}.
Chv\'atal et al.~\cite{MR1927566} proved the conjecture for graphs not containing induced cycles of length at least $\ell+4$ (in which case any avoidable induced path of length $\ell$ is simplicial).
Bonamy et al.~\cite{BDHT19} recently proved the conjecture in general.
Using a similar approach we strengthen their result further in Theorem~\ref{thm:main} (the case of induced paths).

Our results can be stated in terms of combinatorial reconfiguration.
We consider a reachability problem in which the states are walks of a fixed type and length in a graph, the transformations are corresponding shifts, and the target set consists of avoidable walks of the same type and length.
Several other results on reconfiguration of paths are known in the literature.
For example, Demaine et al.~\cite{MR3992972} proved that the reachability problem for shifting paths (``Given two paths in a graph, can one be transformed into the other one by a sequence of shifts?'') is \textsf{PSPACE}-complete.
For shortest $u,v$-paths where each transformation consists in changing a single vertex, the same result was obtained by Bonsma~\cite{MR3122210}.

\subsection{Preliminary definitions and notation}\label{sec:prelim}

Given a vertex $v\in V(G)$ we use standard notations; $N(v)$ and $N[v]$ stand for its open and closed neighborhood, respectively, and $G-v$ denotes the graph obtained from $G$ by removing a vertex $v$. The \emph{order} of a graph $G$ is the number of vertices in $G$. We denote the graph obtained from $G$ by contracting an edge $uv\in E(G)$ by $G/_{uv}$. After such a contraction, it will sometimes be useful to label the newly obtained vertex. We do this by writing $G/_{uv\to u'}$, where $u'$ is the new vertex corresponding to the contracted edge $uv$ in $G$.

Given two graphs $G$ and $H$, their  \emph{Cartesian product} $G \square H$ is the graph with vertex set  $V(G) \times V(H)$, where two vertices $(u,u')$ and $(v,v')$ are adjacent if and only if either
\begin{enumerate*}[label=(\roman*)]
    \item $u = v$ and $u'$ is adjacent to $v'$ in $H$, or
    \item $u' = v'$ and $u$ is adjacent to $v$ in $G$.
\end{enumerate*}

\subsection{Structure of the paper}

In Section~\ref{sec:trails}, we give examples of graphs containing trails of various lengths that do not contain any avoidable trails of the same length. Similar examples for isometric paths are constructed in Section~\ref{sec:isometric-paths}. In Section~\ref{sec:induced-paths} we derive our most important result, stating that every induced path in a graph can be shifted to an avoidable one. The analogous result for paths is proved in two different ways in Section~\ref{sec:paths}. For completeness, we also include the corresponding easy observations about walks in Section~\ref{sec:walks}. We conclude with some open problems in Section~\ref{sec:open}.

\section{Trails}\label{sec:trails}

In this section we will show that Theorem \ref{thm:main} does not extend to the case of trails.
We construct several counterexamples for various lengths $\ell$ of a trail.

For $\ell = 0$ consider the graph $G$ consisting of two vertices $u$ and $v$ joined by an edge, and having a loop at each of $u$ and $v$.  Then, every trail of length $0$ has a unique extension in $G$ (up to reversing the extension) and this extension is not closable. Thus no trail of length $0$ is avoidable in $G$.

Now consider an odd integer $\ell\ge 1$ and let $G_\ell$ be the graph consisting of two vertices and $\ell+2$ parallel edges between them. Then, up to isomorphism there exists a unique trail of length $\ell$ in  $G$. Furthermore, this trail has a unique extension in $G$ and this extension cannot be closed.

For $\ell=2$ consider the graph $G=K_4$. It is easily seen that up to isomorphism
there exists a unique trail of length $\ell$ in $G$.
Furthermore, this trail has exactly three extensions (see Fig.~\ref{pic:K4}),
two of which (those depicted in Fig.~\ref{pic:K4}(b,c)) cannot be closed.
\begin{figure}[!h]
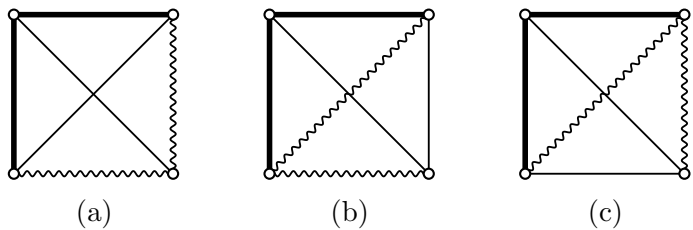

\centering
\begin{tabular}{c@{\hskip1cm}c@{\hskip1cm}c}
  \mpfile{graphs}{51}&\mpfile{graphs}{52}&\mpfile{graphs}{53}
  \\
(a) & (b) & (c)
\end{tabular}
\caption{Thick lines: edges of the trail; wavy lines: edges of an
  extension; ordinary lines: the remaining edges of the graph}\label{pic:K4}
\end{figure}

To get further examples in the class of simple graphs, consider a positive integer $j$, let $\ell = 4j-1$, and let $G$ be the complete bipartite graph $K_{2, 2j+1}$. Then again, up to isomorphism there exists a unique trail of length $\ell$ in $G$ and
its unique extension in $G$ cannot be closed.

\section{Isometric paths}\label{sec:isometric-paths}

This case is not covered by our main theorem (Theorem~\ref{thm:main}), as the following result shows.

\begin{thm}
\label{thm:isometric}For every non-negative integer $\ell$, there exists
a graph $G_{\ell}$ that contains an isometric path of length $\ell$ but
contains no avoidable isometric path of length $\ell$.
\end{thm}

\begin{proof}
For $\ell=0$, let $G_{0}=W_6$ be the wheel on 7 vertices, that is, the
graph obtained from the cycle $C_{6}$ by adding a universal vertex
(see Figure \ref{pic:iso-graphs}(a)). We claim that every vertex of $G_{0}$ is an isometric path of length 0 that is not avoidable. Indeed, since every vertex extends to an isometric
path of length 2, it is enough to show that no isometric path of length
2 in $G_{0}$ is closable. However, this follows from the fact that
$G_{0}$ contains a unique induced cycle of length greater than $3$,
namely the $C_{6}$, which is not isometric.

\begin{figure}[!h]
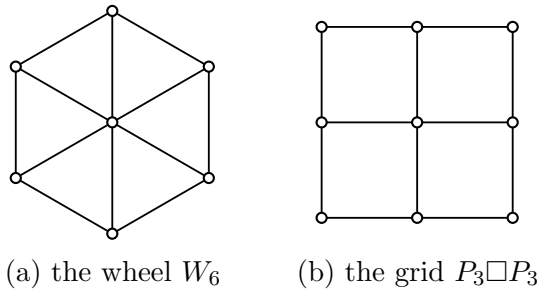

\centering
\begin{tabular}{c@{\hskip1cm}c}
  \mpfile{graphs}1&\raisebox{6pt}{\mpfile{graphs}2}\\
(a) the wheel $W_6$& (b) the grid $P_3\Box P_3$
\end{tabular}
\caption{The cases $\ell=0$ and $\ell=1$}\label{pic:iso-graphs}
\end{figure}

For $\ell=1$ we give two examples: a small specific example and a similar one that is the smallest member of an infinite family of examples for all $\ell\ge 1$. The first one is the graph $G_{1}\cong P_{3}\Box P_{3}$ (see Figure \ref{pic:iso-graphs}(b)).
In this case every edge of $G_{1}$ is an isometric path of length 1 that is not avoidable.
Indeed, since every edge extends to an isometric path of length 3 (see Fig.~\ref{pic:P3P3}), it is enough to show that no isometric path of length 3 in $G_{1}$ is closable. However, this follows from the fact that $G_{1}$ contains a unique induced cycle of length greater than $4$. This cycle is of length 8 and is not isometric.

\begin{figure}
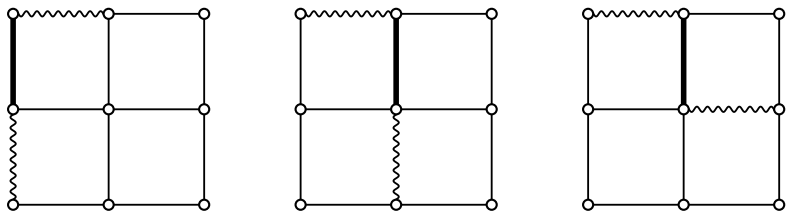

\centering
\begin{tabular}{c@{\hskip1cm}c@{\hskip1cm}c}
  \mpfile{graphs}{21}&\mpfile{graphs}{22}&\mpfile{graphs}{23}
\end{tabular}
\caption{Isometric extensions of edges in $G_1$}\label{pic:P3P3}
\end{figure}

For $\ell\ge1$, let $G_{\ell}$ be any graph of the form $P_{n}\Box C_{n}$
where $n$ is an odd integer greater then $2\ell+4$. We denote vertices
of each factor by numbers from $[n]$, so every vertex
of $G_{\ell}$ is of the form $(i,j)$ for $i,j\in[n]$.
We start by characterizing sufficiently short isometric paths in $G_{\ell}$.
The following claim is implicit in \cite[Chapter 12]{IKR08}.

\medskip
\noindent \textbf{Claim 1.}
Let $P=(v_{0},\ldots,v_{k})$
with $k\le \ell+2$ be a path in $G_{\ell}$. Then $P$ is isometric in $G_{\ell}$ if and only if
for both coordinates the following implication holds: if two vertices of $P$ have the same value of the coordinate, then so does every vertex between them.
\begin{proof}
Suppose that $P$ is isometric in $G_{\ell}$. Take two vertices of $P$ with the same value of some coordinate.
Then there exists a unique shortest path between them in $G_{\ell}$, since $k\le \ell+2<n/2$.
So all edges of this path should belong to $P$.

For the opposite direction, let $u$ and $v$ be two vertices in $P$, and let $Q$ be the $u,v$-path contained in $P$.
We want to show that $Q$ is a shortest $u,v$-path in $G_{\ell}$.
Let $X$ and $Y$ denote the sets of values taken by the first and second coordinates of vertices in $Q$, respectively.
Since $\max\{|X|,|Y|\}\le \ell+3$ and $n\ge 2\ell+5$, the subgraph $G'$ of $G_{\ell}$ induced by $X\times Y$ is isometric in $G_{\ell}$.
Furthermore, if $P$ satisfies the condition from the claim, then the same condition holds for $Q$.
This implies that both coordinates are monotone along $Q$, so $Q$ is a shortest $u,v$-path in $G'$.
It follows that $Q$ is also a shortest $u,v$-path in $G_{\ell}$.
\end{proof}
To complete the proof we show that every isometric path of length
$\ell$ in $G_{\ell}$ has an extension that is not closable. Claim~1
implies that each isometric path $P$ of length $\ell$ has an isometric extension $Q$ that is not constant in the first coordinate (see Fig.~\ref{pic:PnCn}). Such a path $Q$ can only be contained in isometric cycles of length at least $2\ell+4>4$.
We complete the proof by the following claim.

\medskip
\noindent \textbf{Claim 2.} The only isometric cycles in $G_{\ell}$
are the cycles of length $n$ that are constant in the first coordinate
and the cycles of length 4.
\begin{proof}
It is easily seen that mentioned cycles are isometric.

\begin{figure}
\centering
\begin{minipage}{0.45\textwidth}
\centering
\begin{tabular}{c@{\hskip1.5cm}c}
  \mpfile{graphs}{32}&\mpfile{graphs}{31}\\
(a) $x=1$ & (b) $x > 1$
\end{tabular}
\caption{Two relevant cases for constructing path $Q$ in the case when the vertices of $P$ agree in the first coordinate, with common value $x$.}\label{pic:PnCn}
\end{minipage}\quad\quad\quad\quad
\begin{minipage}{0.45\textwidth}
  \centering
  \mpfile{graphs}{4}
  \caption{Situation in the proof of Claim 2.}\label{pic:C4}
  \vspace*{19.8pt}
\end{minipage}
\end{figure}

For the converse direction, consider an isometric cycle $C$ in $G_{\ell}$
that is not constant in the first coordinate. We will show that $C$
is of length 4. Let $a\in[n]$ be the maximal value that
appears as the first coordinate of some vertex in $C$. Let $b\in[n]$
be the minimal value such that $\{ (a-1,b),(a,b)\} $ is
an edge of $C$. We may assume w.l.o.g.~that vertices $v_{0},v_{1},\dots$
of $C$ appear in cyclic order so that $v_{0}=(a-1,b)$ and $v_{1}=(a,b)$.

See Fig.~\ref{pic:C4}. Let $v_{i}=(a,c)$ be the vertex of $C$ having first coordinate $a$ such that $i$ is maximized.
Then, $v_{i+1}=(a-1,c)$ by the maximality of $a$ and $i$. Note
also that $i>1$ and hence $c>b$. By Claim~1, $v_{1},\dots,v_{i}$
are the only vertices in $C$ that maximize the first coordinate.
Cycle $C$ contains a shortest $v_{1},v_{i}$-path in $G_{\ell}$.
Since $n$ is odd, such a path is unique. Similarly, the shortest
$v_{0},v_{i+1}$-path in $G_{\ell}$ is contained in $C$, hence $v_{i+2}=(a-1,c-1)$.
Furthermore, $v_{i-1}=(a,c-1)$, which implies that $\{v_{i-1},v_{i+2}\}$
is an edge of $G_{\ell}$, hence it is also an edge of $C$ since $C$
is isometric. Therefore $v_{i+2}=v_{0}$ and $C$ is of length 4.
\end{proof}
This concludes the proof of Theorem \ref{thm:isometric}.
\end{proof}

\begin{rem}
We leave it to the careful reader to explain why our main construction works only for $\ell\ge1$ and for odd $n$, but not for the case $\ell=0$ or for the case when $n$ is even, and also why one could not replace $P_{n}\Box C_{n}$ by $P_{n}\Box P_{n}$ or $C_{n}\Box C_{n}$.
\end{rem}

\section{Induced paths}\label{sec:induced-paths}

The main result of this section is the following theorem, which settles the case of induced paths from Theorem~\ref{thm:main}.

\begin{thm}
\label{thm:avoidable-induced}
Every induced path in a graph $G$ can
be shifted to an avoidable one.
\end{thm}

We prove Theorem~\ref{thm:avoidable-induced} by adapting the approach used by Bonamy et al.~\cite{BDHT19} to prove that for every positive integer $k$, every graph that contains an induced $P_{k}$ also contains an avoidable induced $P_{k}$ (case $ \3$ of Corollary~\ref{cor:main}).

We first fix some notation. We denote
a path or a cycle simply by a sequence of vertices, e.g., $P=p_{1}\ldots p_{k}$.
Correspondingly, for such a path $P$ and a vertex $x$ not on $P$ we will denote by $xP$ the sequence $xp_{1}\ldots p_{k}$ (which will typically be a path) and by $Px$ the sequence
$p_{1}\ldots p_{k}x$. Thus, if $P'$ is an extension of $P$, then there exist two vertices $x$ and $y$ not on $P$ such that $xPy$ is an extension of $P$. We often use the fact that for an induced subgraph $G'$ of a graph $G$ and two induced paths
$Q_{1}$ and $Q_{2}$ in $G'$, we have $Q_{1}\S{G}{~}Q_{2}$ whenever $Q_{1}\S{G'}{~}Q_{2}$.

We adapt the approach of~\cite{BDHT19} to shifting.
For a graph $G$ and a positive integer $k$, we say that:
\begin{itemize}
\item property $\Hb{G,k}$ holds if every induced path $P_{k}$ in $G$ can be shifted to an avoidable {induced} path;
\item for a vertex $v\in V(G)$, property $\Hr{G,k,v}$ holds if every induced path $P_{k}$ in $G-N[v]$ can be shifted in $G-N[v]$ to an avoidable {induced} path in $G$;
\item property $\Hr{G,k}$
holds if for every $v\in V(G)$ we have $\Hr{G,k,v}$.
\end{itemize}

\begin{lem}
\label{lem:Hr-Hb}$\Hr{G,k}$ implies $\Hb{G,k}$.
\end{lem}

\begin{proof}
Assume $\Hr{G,k}$ and let $Q=q_{1}\ldots q_{k}$ be an induced $P_{k}$
in $G$. If $Q$ is simplicial, then we are done, so assume that $xQy$
is an extension of $Q$ and define $Q'\coloneqq q_{2}\ldots q_{k}y$.
It is clear that $Q\S{G}{~}Q'$. Furthermore, by $\Hr{G,k,x}$ the path
$Q'$ can be shifted in $G-N[x]$ to a path $Q^{*}$ that
is avoidable in $G$. But then $Q\S{G}{~}Q'\S{{G-N[x]}}{~}Q^{*}$,
and hence $Q\S{G}{~}Q^{*}$. Since $Q$ was arbitrary, this shows $\Hb{G,k}$.
\end{proof}

We need the following result, which is implicit in the proof of \cite[Lemma 15]{BDHT19}.
\begin{lem}
\label{lem:contraction}Let $G$ be a graph, let $uv\in E(G)$, let
$G'\coloneqq G/_{uv\to u'}$ and let $P$ be an induced path in $G'-N[u']$.
Then $P$ is avoidable in $G$ whenever it is avoidable in $G'$.
\end{lem}

\noindent For the sake of completeness we include the proof.
\begin{proof}
Since $G'-N[u']=G-N[\{ u,v\}]$,
the path $P$ is an induced path in $G$. Suppose that $P$ is avoidable
in $G'$ and consider an extension $xPy$ of $P$ in $G$. Since $P$
is contained in $G'-N[u']$, vertices $x$ and $y$ are
distinct from $u$ and $v$. Therefore, $xPy$ is an induced path
in $G'-u'$. Since $P$ is avoidable in $G'$, there exists an induced
cycle $C$ in $G'$ containing $xPy$. If $C$ does not contain $u'$,
then $C$ is also induced in $G$. Otherwise, replacing $u'$ in $C$
with either $u$, $v$, $uv$, or $vu$ as appropriate, we obtain
an induced cycle in $G$ containing $xPy$. This shows that $P$ is
avoidable in $G$.
\end{proof}


\begin{lem}
\label{lem:Hr-forall}For any graph $G$ and positive integer $k,$
property $\Hr{G,k}$ holds.
\end{lem}

\begin{proof}
Fix $k$ and let $G$ be a graph of minimal order for which $\Hr{G,k}$
does not hold. In particular, let $u$ be a vertex in $G$ such that
$\Hr{G,k,u}$ does not hold. Then, there exists an induced path $Q$
in $G-N[u]$ that cannot be shifted in $G-N[u]$
to any avoidable path in $G$.

Since $G-N{[u]}$ is of smaller order than $G$, property
$\Hr{G-N{[u]},k}$ holds. By Lemma~\ref{lem:Hr-Hb}, property
$\Hb{G-N{[u]},k}$ holds as well. Therefore, there exists
a path $Q'=q_{1}\dots q_{k}$ such that $Q'$ is avoidable in $G-N[u]$
and $Q\S{G-N[u]}{~}Q'$. The choice of $Q$ implies that $Q'$
is not avoidable in $G$, thus $Q'$ has an extension $xQ'y$ that
is not closable in $G$. Note that precisely one of $x,y$ is a member
of $N(u)$, as otherwise the extension $xQ'y$ would be
closable in $G$. We may assume w.l.o.g. that $x$ is a common neighbor
of $u$ and $q_{1}$.

Set $G'\coloneqq G/_{ux\to u'}$.
Observe that $Q''\coloneqq q_{2}\dots q_{k}y$ does not contain $u$, $x$, or any neighbor in $G$ of $u$ or $x$. Therefore, $Q''$ is a path in $G'-N[u']$.
Again, the minimality of $G$ implies
property $\Hr{G',k}$, in particular, also $\Hr{G',k,u'}$ holds.
Hence, $Q''$ can be shifted in $G'-N[u']$ to an induced
path $Q^{*}$ that is avoidable in $G'$. So we have $Q\S{G-N[u]}{~}Q'\S{G-N[u]}{~}Q''\S{G'-N[u']}{~}Q^{*}$,
where the relations follow from the definitions of $Q',Q''$, and
$Q^{*}$, respectively. Since $G'-N[u']$ is an induced
subgraph of $G-N[u]$, we have $Q\S{G-N[u]}{~}Q^{*}$.
The choice of $Q$ implies that $Q^{*}$ is not avoidable in $G$,
which contradicts Lemma~\ref{lem:contraction}.
\end{proof}

\begin{proof}[Proof of Theorem~\ref{thm:avoidable-induced}]
 Immediate from Lemmas~\ref{lem:Hr-Hb} and \ref{lem:Hr-forall}.
\end{proof}

The proof of Theorem~\ref{thm:avoidable-induced} is constructive.
It gives an algorithm for computing a sequence of shifts transforming a given induced path in a graph $G$ to an avoidable one, see {Procedures~\ref{alg:shifting} and~\ref{alg:refined}}.
We do not know if the algorithm runs in polynomial time.

\begin{algorithm}[ht!]
\begin{algorithmic}[1]

\Require a graph $G$ and an induced path $P=p_{1}p_{2}\dots p_{k}$ in $G$

\Ensure a sequence $S$ of paths shifting $P$ to an avoidable induced path in $G$

\If{there exists an extension $xPy$ of $P$}

\State $Q\gets yp_k\ldots p_1x$

\State \Return $P,\textsc{RefinedShifting}(G,Q)$

\Else

\State \Return $P$

\EndIf

\end{algorithmic}

\caption{\label{alg:shifting}$\textsc{Shifting}(G,P)$}
\end{algorithm}

\begin{algorithm}[ht!]
\caption{\label{alg:refined}$\textsc{RefinedShifting}(G,P)$}
\begin{algorithmic}[1]
\Require a graph $G$ and an induced path $P=p_{1}\dots p_{k+2}$ in $G$
\Ensure a sequence $S$ of paths in $G-N[p_{k+2}]$ shifting $p_1\ldots p_k$ to an avoidable induced path in $G$

\State $P' \gets p_{1}\dots p_{k}$
\State $S\gets$ the one-element sequence containing path $P'$
\If{there exists an extension $xP'y$
in $G-N[p_{k+2}]$}
\State $S\gets S,\textsc{RefinedShifting}(G-N[p_{k+2}],xP'y)$
\EndIf
\State $Q\gets$ the end path of $S$

\If{$Q$ has an extension $xQy$ in $G$ such that
$y$ is the unique neighbor of $p_{k+2}$ in $\{ x,y\}$}
\State let $Q=q_{1}\dots q_{k}$ such that $y$ is adjacent to $q_{k}$
\State $Q'\gets xq_{1}\dots q_{k}$
\State  $G'\gets G/_{p_{k+2}y\rightarrow y'}$
\State  $S'\gets\textsc{RefinedShifting}(G',Q'y')$
\State \Return $S,S'$
\Else
\State \Return $S$
\EndIf
\end{algorithmic}
\end{algorithm}

\section{Paths}\label{sec:paths}

The main result of this section is the following theorem, which settles the case of paths from Theorem~\ref{thm:main}.

\begin{thm}\label{thm:paths}
Every path in a graph $G$ can be shifted to an avoidable one.
\end{thm}

We offer two proofs. The first proof will rely on several observations about line graphs.
Recall that the \emph{line graph} of a graph $G$ is the graph $G'$ with $V(G') = E(G)$ such that two distinct edges $e$ and $f$ of $G$ form a pair of adjacent vertices in $G'$ if and only if $e$ and $f$ share an endpoint in $G$.

\begin{lem}\label{lem:line-graphs}
Let $G$ be a graph and let $G'$ be its line graph. Then the following statements hold.
\begin{enumerate}[(a)]
\item\label{item-1} Let $P$ be a path  of length $\ell\ge 1$ in $G$ and let $P'$ be the sequence of edges of $P$ along the path. Then $P'$ is an induced path of length $\ell-1$ in $G'$.
\item\label{item-2} Let $C'$ be an induced cycle of length at least four in $G'$. Then, the sequence of vertices of $C'$ along the cycle yields a sequence of edges of $G$ that forms a cycle $C$ in $G$.
\item\label{item-2.5} Let $P'$ be an induced path in $G'$ and let $\ell$ be the length of $P'$.
Then, the sequence of vertices of $P'$ along the path yields a sequence of edges of $G$ that forms a path $P$ of length $\ell+1$ in $G$.
\item\label{item-3} For every $\3$-avoidable induced path $P'$ in $G'$, the corresponding path $P$ in $G$ (as in~\ref{item-2.5}) is a $\2$-avoidable path in $G$.
\item\label{item-4}
For every two induced paths $P'$ and $Q'$ in $G'$  that are $\3$-shifts of each other in $G'$, the corresponding paths $P$ and $Q$ in $G$ (as in~\ref{item-2.5}) are $\2$-shifts of each other in~$G$.
\end{enumerate}
\end{lem}

For the sake of completeness we include a proof, which is lengthy but straightforward.

\begin{proof}
\ref{item-1}.
Let $e_1,\ldots, e_{\ell}$ be the edges of $P$ in order.
Since $P$ is a path in $G$, these edges are pairwise distinct.
Furthermore, for all $i,j\in \{1,\ldots, \ell\}$ with $i<j$, edges $e_i$ and $e_j$ share an endpoint in $G$ if and only if $j = i+1$; thus, $e_i$ and $e_j$ are adjacent as vertices of $G'$ if and only if $j = i+1$.
We conclude that $P'$ is an induced path of length $\ell-1$ in $G'$.

\ref{item-2}. Let $\ell\ge 4$ be the length of $C'$ and let $e_1,\ldots, e_\ell$ be a cyclic order of vertices of $C'$.
Then $e_1,\ldots, e_\ell$ are pairwise distinct edges of $G$, with two sharing an endpoint in $G$ if and only if they appear consecutively in the cyclic order.
In particular, since $\ell\ge 4$, no three of these edges share a common endpoint.
Thus, if for all $i\in \{1,\ldots, \ell\}$ we denote by $v_{i}$ the common endpoint in $G$ of $e_i$ and $e_{i+1}$ (indices modulo $\ell$), then vertices $v_1,\ldots, v_\ell$ are pairwise distinct, and  $e_i = \{v_{i-1},v_i\}$ for all $i\in \{1,\ldots, \ell\}$ (with $v_0 = v_\ell$).
In particular, $C = (v_1,e_1,v_2,\ldots, v_{\ell},e_{\ell},v_1)$ is a cycle in $G$ formed by the edges of $C'$.

\ref{item-2.5}. The proof is very similar to (but simpler than) that of item~\ref{item-2}.

\ref{item-3}. Let $e_1,\ldots, e_{\ell+1}$ be the vertices of $P'$ in order.
By~\ref{item-2.5}, the sequence of edges $e_1,\ldots, e_{\ell+1}$ forms a path $P$ of length ${\ell+1}$ in $G$.
Suppose that $P'$ is an $\3$-avoidable induced path in $G'$.
To show that $P$ is a $\2$-avoidable path in $G$, we verify that every $\2$-extension of $P$ is $\2$-closable.
Let $Q$ be an arbitrary $\2$-extension of $P$ in $G$.
Then there exist two edges $e_0$ and $e_{\ell+2}$ in $G$ such that $Q$ is a path of length $\ell+3$, with edges $e_0,e_1,\ldots, e_{\ell+1},e_{\ell+2}$ in order.
By part~\ref{item-1} of the lemma, this sequence of edges is a sequence of vertices in $G'$ forming an induced path $Q'$ of length $\ell+2$ in $G'$.
Note that $Q'$ is an $\3$-extension of the induced path $P'$ in $G'$.
Since $P'$ is an $\3$-avoidable induced path in $G$, every $\3$-extension of $P'$ is $\3$-closable.
In particular, $Q'$ is contained in an induced cycle $C'$ in $G'$.
Since $Q'$ is an induced path contained in $C'$, the length of $C'$ is at least $(\ell+2)+2\ge 4$.
Thus, by part~\ref{item-2} of the lemma, the sequence of vertices of $C'$ along the cycle yields a sequence of edges of $G$ that forms a cycle $C$ in $G$.
Furthermore, $Q$ is contained in $C$ and hence $\2$-closable.
Thus, every $\2$-extension of $P$ is closable and $P$ is indeed a $\2$-avoidable path in~$G$.

\ref{item-4}. Let $\ell$ be the common length of the paths $P'$ and $Q'$.
Then $P$ and $Q$ are both of length $\ell+1$.
The paths $P'$ and $Q'$ are $\3$-shifts of each other in $G'$, and hence, considering paths as subgraphs, the union of $P'$ and $Q'$ is an induced path $R'$ of length $\ell+1$ in $G'$.
Let $R$ be the path in $G$ corresponding to $R'$ (as in item~\ref{item-2.5} of the lemma).
Then $R$ is a path of length $\ell+2$ in $G$ that is the union of paths $P$ and $Q$.
This shows that $P$ and $Q$ are $\2$-shifts of each other in~$G$.
\end{proof}

\begin{proof}[First proof of Theorem~\ref{thm:paths}.]
The first proof is based on a reduction to Theorem~\ref{thm:avoidable-induced}.
Let $P$ be a path in $G$ and let $\ell$ be the length
of $P$. Suppose that $\ell=0$. Then $P$ corresponds to a vertex $v\in V(G)$.
Let $U$ be the connected component of $G$ containing $v$. If $U$
contains only $v$ then clearly $P$ is avoidable in $G$. Otherwise,
let $u$ be a vertex in $U$ such that $U-u$ is connected. (Such
a vertex exists, for example, take a leaf of a spanning tree in $U$.)
Then $u$ is an avoidable path in $G$ such that $P\S{G}{\2}u$ .

Suppose now that $\ell\ge1$ and let $G'$ be the line graph of $G$.
Let $P'$ be the sequence of edges of $P$.
By item~\ref{item-1} of Lemma~\ref{lem:line-graphs}, $P'$ is an induced path of length $\ell-1$ in $G'$.
By Theorem~\ref{thm:avoidable-induced} there
exists an induced path $Q'$ that is avoidable in $G'$ and such that
$P'\S{G'}{\3}Q'$.
By items~\ref{item-2.5} and~\ref{item-3} of Lemma~\ref{lem:line-graphs}, the sequence of vertices of $Q'$ in $G'$ corresponds to a sequence of edges in $G$ forming a path $Q$ that is avoidable in $G$.
Furthermore, since $P'\S{G'}{\3}Q'$, we conclude using item~\ref{item-4} of Lemma~\ref{lem:line-graphs} that $P\S{G}{\2}Q$.
\end{proof}

In the second proof, all our arguments on paths will only depend on the corresponding sequences of vertices, even in the case of graphs with blue edges. Thus, we use notation introduced in Section~\ref{sec:induced-paths} and represent each path simply as a sequence of vertices.

\begin{proof}[Second proof of Theorem~\ref{thm:paths}.]
The second proof works directly on $G$ and is based on properties of depth-first search (DFS) trees. Let $P$ be a path in $G$ and let $\ell$ be the length of $P$. Consider a DFS traversal of $G$ starting in $P$ and let $T$ be the corresponding DFS tree.
Let $Q$ be a longest root-to-leaf path in $T$ such that $P$ is a subpath of $Q$. We shift $P$ along $Q$ all the way to the last vertex of $Q$, obtaining this way a path $P' = v'_0v'_1\ldots v'_\ell$, where $v'_\ell$ is a leaf in $T$. Let $Q'$ be a longest root-to-leaf path in the subtree of $T$ rooted at $v_0'$.
We now define a path $P''=v''_0v''_1\ldots v''_\ell$ depending on the length of $Q$.
If the length of $Q$ is at least~$2\ell$ we set $P''=P'$ (see Fig.~\ref{pic:path-tree}a).
Otherwise we shift $P'$ to the subpath $P''$ of $Q'$ such that $v''_\ell$
is a leaf in $T$ (see Fig.~\ref{pic:path-tree}b).

\begin{figure}[h]
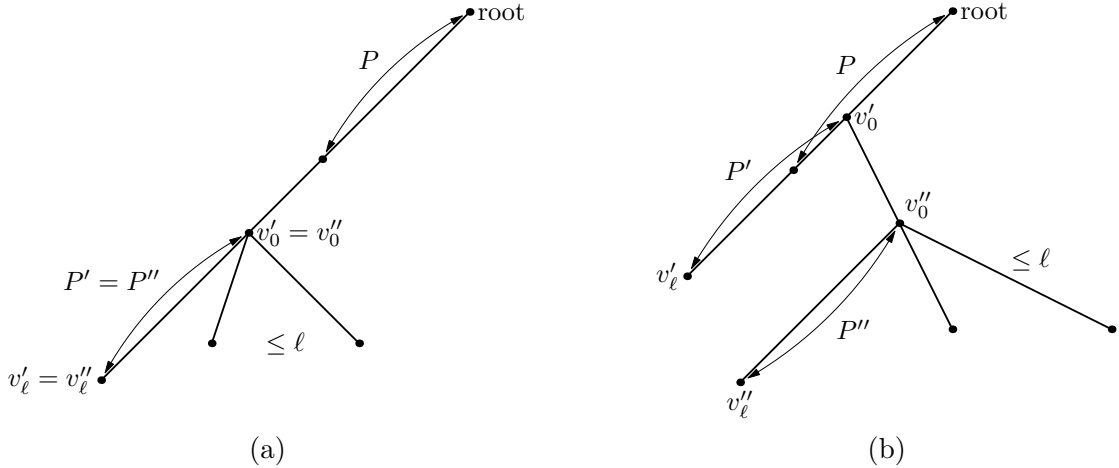

\centering
\begin{tabular}{c@{\hskip1.62cm}c}
  \raisebox{8.73pt}{\mpfile{trees}{11}} &\mpfile{trees}{1}\\
  (a) & (b)
\end{tabular}
  \caption{
  The two cases from the second proof, depending on the length of $Q$
  }\label{pic:path-tree}
\end{figure}
Note that the length of each path from $v''_0$ to a~leaf of the subtree of $T$ rooted at $v_0'$ is at most $\ell$, since  $Q'$ is a longest root-to-leaf path in this subtree.

If $P''$ is avoidable in $G$, we are done. Otherwise, $P''$ has an extension $$xP''y = xv''_0v''_1\ldots v''_\ell y$$ that is not closable. Since $T$ is a~DFS tree in $G$, all neighbors of $v''_\ell$ in $G$ are ancestors of $v''_\ell$ in $T$. In particular, this implies that $y$ is an ancestor of $v''_\ell$ and hence also an ancestor of $v''_0$.
Note that $y$ is a vertex of the path  $Q''=r\dots v'_0\dots v''_0$, where $r$ is the root of $T$ (see Fig.~\ref{pic:bad-ext} for the case when the length of $Q$ is less than $2\ell$).

\begin{figure}[ht]
\centering
  \mpfile{trees}{2}
  \caption{A hypothetical non-closable extension of $P''$}\label{pic:bad-ext}
\end{figure}

Since $xP''y$ is not closable, we infer that $x$ is not an ancestor of $v_0''$ in $T$.
Thus, $x$ is a child of $v''_0$ in $T$. Let $Q'''= v''_0x\dots w$ be a path in $T$ such that $w$ is a leaf in $T$.
We now shift $P''$ following  $Q'''$ from $v''_0$ to the last vertex of $Q'''$, obtaining this way a path $P'''$, the last vertex of which is $w$.  Note that, by choice  of $P''$, vertex $v''_0$ belongs to the path $P'''$ (see Fig.~\ref{pic:bad-ext}). Thus, if $w$ has a neighbor in $G$ that is a proper ancestor of $v''_0$ in $T$, then $xP'y$ would be a closable extension of $P''$, which is not possible. We conclude that all neighbors of $w$ in $G$ are also vertices of $P'''$. Hence, $P'''$ is a simplicial path in $G$.
Since $P\S{G}{\2}P'$, $P'\S{G}{\2}P''$, and $P''\S{G}{\2}P'''$, we have
$P\S{G}{\2}P'''$. Thus, $P$ can always be shifted to an avoidable path in $G$.
\end{proof}

The second proof of Theorem~\ref{thm:paths} gives a polynomial-time algorithm for computing a sequence of shifts transforming a given path in a graph $G$ to an avoidable one, see Procedure~\ref{alg:shifting-simplepath}.

\begin{algorithm}[ht!]
\begin{algorithmic}[1]

\Require a graph $G$, a path $P$ in $G$

\Ensure a sequence $S$ of paths shifting $P$ to an avoidable path
in $G$

\State $\ell\gets\textsc{Length}(P)$

\State $T\gets \textsc{DFS}(G,P)$ \Comment{DFS tree w.r.t.~an ordering
  starting from $P$}

\State $Q\gets \textsc{Longest}(T,P)$ \Comment{A longest root-to-leaf path in $T$ starting with $P$.}

\State $S \gets \textsc{ShiftAlong}(Q,P)$ \Comment{The sequence of shifts
  along the path $Q$}

\State $P'\gets S[-1]$ \Comment{The last path in the sequence}

\State $v'_0 \gets P'[1]$ \Comment{The first vertex in the path}

\State $Q'\gets \textsc{Longest}(T, v'_0)$ \Comment{A longest
 root-to-leaf path in the subtree of $T$ rooted at $v'_0$}

\If{$\textsc{Length}(Q')\leq \ell$}

\State $P''\gets P'$
\Else

\State $R'\gets \textsc{Reverse}(P'), Q'$

\State $S \gets S, \textsc{ShiftAlong}(R',\textsc{Reverse}(P'))$

\State $P''\gets S[-1]$ \Comment{The last path in the sequence}

\EndIf

\If{there exists an extension $xP''y$ of $P''$ which is not closable}

\State $v''_0\gets P''[1]$

\State $Q'''\gets \textsc{Longest}(T, v''_0x)$ \Comment{A longest
path to the leaf starting with the edge $v''_0x$}

\State $R''\gets \textsc{Reverse}(P''), Q'''$

\State \Return $S, \textsc{ShiftAlong}(R'',\textsc{Reverse}(P''))$

\Else
\State \Return $S$

\EndIf

\end{algorithmic}

\caption{\label{alg:shifting-simplepath}$\textsc{PathShifting}(G,P)$}
\end{algorithm}

\section{Walks}\label{sec:walks}

For this case we provide two simple observations.
The first one already suffices to prove the first claim of Theorem \ref{thm:main} and the case $\0$ of Corollary \ref{cor:main}.

\begin{observation}\label{obs:walks-1}
Every walk in a graph is avoidable.
\end{observation}

\begin{proof}
Indeed, any extension $W'$
of a walk
$W$ is a subwalk of the closed walk obtained by traversing $W'$ first
in one direction and then in the opposite one.
\end{proof}

Furthermore, if the graph is connected, then any walk can be shifted to any walk of the same length.

\begin{observation}\label{obs:walks-2}
Let $W$ and $W'$ be two walks of the same length $\ell$ in a connected graph $G$. Then, $W$ can be shifted to $W'$.
\end{observation}

\begin{proof}
Let $W^*$ be the concatenation of walks $W$, $W''$, and $W'$, where
$W''$ is an arbitrary walk in $G$ from the last vertex of $W$ to the first vertex of $W'$.
Clearly $W^*$ is also a walk in $G$, and its subwalks of length $\ell$ form a sequence of walks that shows that $W$ can be shifted to $W'$.
\end{proof}

\section{Open problems}\label{sec:open}

We conclude with the following open problems:
\begin{enumerate}
\item The proof of Theorem~\ref{thm:avoidable-induced} is constructive and produces a sequence $S$ of paths shifting a given induced path $P$ in a graph $G$ to an avoidable induced path. Similarly, our proofs of Theorem~\ref{thm:paths} do the same for the case of not necessarily induced paths.
For the latter case, we believe that with an appropriate compact representation of the output and a suitable implementation of Procedure~\ref{alg:shifting-simplepath} one can achieve linear running time.
On the other hand, about the induced case we know much less.
Given a graph $G$ and an induced path $P$ in $G$, is there a polynomial upper bound on the minimum length of a sequence of shifts transforming $P$ to an avoidable induced path and, if so,
can a sequence of polynomial length be computed efficiently?
In particular, does the algorithm given by the proof of Theorem~\ref{thm:avoidable-induced} (Procedures~\ref{alg:shifting} and~\ref{alg:refined}) run in polynomial time?

\item For a positive integer $k$, what are the graphs that have an avoidable trail of length $k$ whenever they have a trail of length $k$? What are the graphs for which the above property holds for all~$k$?

For a positive integer $k$, what are the graphs in which every trail of length $k$ can be shifted to an avoidable one? What are the graphs in which every trail can be shifted to an avoidable one?

What is the time complexity of recognizing graphs with above properties?

The above questions are also open for isometric paths.

\item In paper~\cite{MR3992972} the problem of determining whether there exists a sequence of shifts from a given path to another one is proved \textsf{PSPACE}-complete, while the computational complexity status of analogous problems for trails, induced paths, and isometric paths remains open.
The corresponding problem for walks is trivial.

\item Let us say that an induced path $P$ in a graph $G$ is \emph{strongly avoidable} if
there exists a component $C$ of $G-N[P]$ such that
every extension of $P$ can be closed to an induced cycle using only vertices of $C$.
It follows from~\cite[Theorem 5.1]{MR1626534} (see also~\cite{MR3303861}) that every graph $G$ has a strongly avoidable $P_1$. For $k>1$, which graphs have strongly avoidable induced paths $P_k$?

\end{enumerate}

\paragraph{Acknowledgments.}
The authors are grateful to two anonymous reviewers for helpful remarks.
The work for this paper was done in the framework of a bilateral project between Slovenia and Russia, financed by the Slovenian Research Agency (BI-RU/19-20-022). The second  and  third authors acknowledge partial support of the Slovenian Research Agency (I0-0035, research programs P1-0285, P1-0297, and P1-0383, and research projects J1-1692, J1-9110, J1-9187, N1-0102, and N1-0160).
The work of the first and fourth authors was done  within the framework of the HSE University Basic Research Program.  The work of the fourth author was partially supported by State Assignment, theme no.~0063-2019-0003.


\end{document}